\documentclass[12pt,reqno]{amsproc}
\usepackage[hyphens]{url} 
\numberwithin{equation}{section}

\makeatletter
\@namedef{subjclassname@2020}{%
	\textup{2020} Mathematics Subject Classification}
\makeatother

\usepackage{amsmath,amssymb,color}
\usepackage{amsthm} 
\usepackage{mathtools}
\usepackage{graphicx}
\usepackage[bookmarksnumbered,colorlinks]{hyperref}
\usepackage{enumitem}
\mathtoolsset{showonlyrefs}
\usepackage[pagewise,mathlines]{lineno}
\theoremstyle{plain} 
\newtheorem{theorem}{\indent\sc Theorem}[section]
\newtheorem{lemma}[theorem]{\indent\sc Lemma}

\theoremstyle{definition} 

\newtheorem{remark}[theorem]{\indent\sc Remark}

\title[A remark on the first eigenvalue of the $p$-Laplacian]{\sc A remark on the first eigenvalue of the $p$-Laplacian on compact submanifolds in the unit sphere}

\bigskip

\author[F.R. dos Santos and M.N. Soares]{F\'abio R. dos Santos$^{\ast}$ and Matheus N. Soares}

\address{
Departamento de Matem\'atica \\
Universidade Federal de Pernambuco \\
50.740-540 Recife, Pernambuco \\
Brazil}
\email{fabio.reis@ufpe.br}
\email{matheus.nsoares@ufpe.br}


\thanks{$^{\ast}$Corresponding author}

\begin{document}
	
\begin{abstract}
An integral inequality for the singular $p$-laplacian is established for $3/2<p<2$. As consequence, lower bounds for the first eigenvalue of the $p$-laplacian are obtained for minimal submanifolds and prescribed scalar curvature submanifolds in the unit sphere.
\end{abstract}
	
\maketitle

\section{Introduction}

Over the past several decades, there has been a discernible surge in scholarly interest directed towards the investigation of the $p$-Laplacian for values of $p$ within the interval $(1,2)$, particularly in its singular case. We recall that the $p$-Laplacian on a compact Riemannian manifold $M^{n}$ is defined as the second order quasilinear elliptic operator
\begin{equation}\label{eq:1.0}
\Delta_{p}u=-{\rm div}(|\nabla u|^{p-2}\nabla u),\quad 1<p<\infty.
\end{equation}

Within the domain of mathematical analysis, this surge in interest has engendered a concomitant exploration into the characterization of functions $f$ such that they satisfy the equation $\Delta_p u=f$. Consequently, a pivotal pursuit in this field pertains to the identification and analysis of functions that fulfill this criterion. Among the myriad attributes under scrutiny, one finds the intriguing properties of monotonicity and symmetry. Notably, these characteristics, along with others, have been the subject of rigorous examination and elucidation in seminal works such as that of Damascelli \cite{Dama:98}. From the perspective of Physics, various applications arise, such as those concerning pseudo-plastic fluids. For instance, one notable application pertains to the study of unsaturated flow for $p=3/2$ (cf. \cite{Diaz:94}) and glaciology for $p\in (1,4/3]$. Notably, within the $3/2<p<2$, Leibenson conducted an investigation into the turbulent filtration of gas in porous mediums (for further details, refer to \cite{Leibenson:45}). Additionally, comprehensive insights into this subject can be found in the work of Benedikt (cf. \cite{Benedikt:18}).

We can consider the eigenvalue problem of $\Delta_{p}$ similarly as the usual laplacian (when $p=2$). We say that a real number $\lambda$ is a {\em Dirichlet eigenvalue} if there exists a non-zero function $u$ satisfying the following equation:
\begin{equation}\label{eq:1.1}
\Delta_{p}u=\lambda|u|^{p-2}u\quad\mbox{in}\quad M^{n}\quad\mbox{and}\quad u=0\quad\mbox{on}\quad\partial M.
\end{equation}
According to \cite{Lindqvist:15}, these numbers $\lambda$ forms a non-increasing sequence such that there exists an isolated minimum eigenvalue called {\em the first eigenvalue} of the $p$-Laplacian (see also~\cite{le:06}). So, the first nontrivial Dirichlet eigenvalue of $M^{n}$ is given by
\begin{equation*}\label{eq:1.2}
\lambda_{1,p}(M)=\inf\left\{\dfrac{\int_{M}|\nabla u|^{p}dM}{\int_{M}|u|^{p}dM}\,;\,u\in W^{1,p}_{0}(M)\backslash\{0\}\right\}.
\end{equation*}

In the field of geometric analysis theory, there is a natural way to connect the geometric features of a Riemannian manifold with the $p$-Laplacian through its eigenvalues. In this setting, Matei~\cite{Matei:00} explored how these elements interact with geometry. In fact, she extended several intrinsic geometric results initially established for the standard Laplacian to encompass the $p$-Laplacian. Notable extensions include adaptations for $p>2$ of Chern's comparison principle for geodesic balls, an expansion of the Faber-Krahn inequality, and a broadening of the Lichnerowicz-Obata theorem. Most results presented by Matei are only established for $p>2$.

On the other hand, in the context of isometric immersions in the unit sphere $\mathbb{S}^{n+q}$, Leung~\cite{Leung:83} established a sharp estimate between the eigenvalues of the usual Laplacian for minimal submanifolds of sphere through an integral inequality. In fact, Leung presented a lower bound for the square length of the second fundamental form $S$ via any eigenvalue of the Laplace-Beltrami operator. Besides that, he showed that spheres are obtained when the equality happens. Years later, Liu and Zhang \cite{Liu:07} proved a similar estimate for arbitrary submanifolds of sphere. 

In a recent contribution~\cite{Santos:23}, the authors undertook an investigation into compact minimal submanifolds within the unit sphere. They introduced a divergence type operator and proceeded to derive Bochner and Reilly formulas pertinent to it. As a consequential application, they derived integral inequalities which encompassed the squared norm of the second fundamental form. Notably, this extension augmented the prior findings of Leung \cite{Leung:83} to encompass the first eigenvalue of the $p$-Laplacian, particularly within the context of manifolds featuring boundaries. It is worth emphasizing that the validity of the results posited by the authors is contingent upon the restriction that $p$ does not fall within the range $1<p<2$. In this paper, our primary objective is extend the results presented in \cite{Santos:23,Santos:24} for $p\in \left(3/2,2\right)$.

\section{Some preliminaries and key lemmas}\label{preliminares}

Let $M^{n}$ be an $n$-dimensional connected compact (with nonempty boundary $\partial M$) submanifold immersed in a unit Euclidean sphere $\mathbb{S}^{n+q}$ with codimension $q$. We will make use of the following convention on the range of indices:
\begin{equation}\label{eq:2.1}
1\leq A,B,C,\ldots,\leq n+q,\quad1\leq i,j,k,\ldots\leq n\quad\mbox{and}\quad n+1\leq \alpha,\beta,\gamma,\ldots\leq n+q.
\end{equation}
We choose a local field of orthonormal frame $\{e_{1},\ldots,e_{n+q}\}$ in $\mathbb{S}^{n+q}$, with dual coframe $\{\omega_{1},\ldots,\omega_{n+q}\}$, such that, at each point of $M^{n}$, $\{e_{i}\}_{i=1}^{n}$ are tangent to $M^{n}$ and $\{e_{\alpha}\}_{\alpha=n+1}^{n+q}$ are normal to $M^{n}$. Let us consider $\{\omega_{B}\}$ the corresponding dual coframe, and $\{\omega_{BC}\}$ the connection $1$-forms on $\mathbb{S}^{n+q}$. Restricting on $M^{n}$, we have 
\begin{equation}
\omega_{\alpha}=0,\quad n+1\leq\alpha \leq n+q,
\end{equation}
and consequently, by Cartan's Lemma we can write
\begin{equation}
\omega_{i\alpha}=\sum_{j}h_{ij}^{\alpha}\omega_{j},\quad h^{\alpha}_{ij}=h^{\alpha}_{ji},
\end{equation}
which gives the second fundamental form $A$ of $M^{n}$,
\begin{equation}
A=\sum_{\alpha,i,j}h_{ij}^{\alpha}\omega_{i}\otimes\omega_{j}\otimes e_{\alpha}.
\end{equation}
Moreover, we define the mean curvature vector field and the mean curvature function by
\begin{equation}
h=\dfrac{1}{n}\sum_{\alpha}H^{\alpha}e_{\alpha}\quad\mbox{and}\quad H=|h|,
\end{equation}
where $H^{\alpha}=\sum_{j}h_{jj}^{\alpha}$. In particular we say that $M^{n}$ has constant mean curvature if the function $H$ is constant. In the case where  this constant is zero, we say that $M^{n}$ is a minimal submanifold of $\mathbb{S}^{n+q}$.

Directly related to the second fundamental $A$ of $M^{n}$ we have the Gauss equation
\begin{equation}\label{eq:2.4}
R_{ijkl}=(\delta_{ik}\delta_{jl}-\delta_{il}\delta_{jk})+\sum_{\alpha}(h_{ik}^{\alpha}h_{jl}^{\alpha}-h_{il}^{\alpha}h_{jk}^{\alpha}).
\end{equation}
From~\eqref{eq:2.4}, we obtain the Ricci curvature tensor and the normalized scalar curvature $R$, respectively, by
\begin{equation}\label{eq:2.5}
R_{ik}=(n-1)\delta_{ik}+n\sum_{\alpha}H^{\alpha}h_{ik}^{\alpha}-\sum_{\alpha,j}h_{ij}^{\alpha}h_{jk}^{\alpha}
\end{equation}
and
\begin{eqnarray}\label{eq:2.6}
R=\dfrac{1}{(n-1)}\sum_{i}R_{ii}.
\end{eqnarray}
By using equations~\eqref{eq:2.5} and~\eqref{eq:2.6}, we get the following relation
\begin{eqnarray}\label{eq:2.7}
n(n-1)R=n(n-1)+n^{2}H^{2}-S,
\end{eqnarray}
where
\begin{equation}\label{eq:2.8}
S=\sum_{\alpha,i,j}(h_{ij}^{\alpha})^{2},
\end{equation}
denotes the squared norm of the second fundamental form.

In order to proof our main results, we need of the following two results. The first one is a low estimate of the Ricci curvature in terms of the squared norm of the second fundamental form and the mean curvature function:
\begin{lemma}{\cite[Main Theorem]{Leung:92}}\label{lem:2.1}
Let $M^n$ be a submanifold of the Riemannian manifold $\mathbb{S}^{n+q}$. Let ${\rm Ric}$ denotes the function that assigns to each point of $M^{n}$ the minimum Ricci curvature. Then
\begin{equation}\label{eq:2.9}
{\rm Ric}\geq-\dfrac{n-1}{n}\left(S+\frac{n(n-2)}{\sqrt{n(n-1)}}H\sqrt{S-nH^2}-n-2nH^2\right).
\end{equation}
\end{lemma}

Before to present the second result, we will recall some facts about isometric immersions with nonempty boundary. Let us consider $\eta$ the outer unit normal field of $\partial M$. We define the shape operator $\mathcal{A}_{\eta}$ and the mean curvature function of $\partial M$ in $M^{n}$, respectively, by
\begin{equation}\label{eq:2.10}
\mathcal{A}_{\eta}(X)=-\nabla_{X}\eta\quad\mbox{and}\quad\mathcal{H}=\dfrac{1}{n-1}{\rm tr}(\mathcal{A}_{\eta}),
\end{equation}
for any $X\in\mathfrak{X}(\partial M)$, where $\nabla$ denotes the Levi-Civita connection of $M^{n}$. In this setting let us recall that a compact manifold is said to have {\em mean convex boundary} if the mean curvature $\mathcal{H}$ is nonpositive on $\partial M$. Let us denote by $\nabla^{\partial}$ and $\Delta^{\partial}$ the covariant derivative and the Laplacian operator on $\partial M$ with respect to the induced Riemannian metric. In this picture, the second key result is a suitable version of~\cite[Proposition 4.1]{Santos:23} for our interest.

\begin{lemma}\label{lem:2.2}
Let $M^{n}$ be a compact manifold with mean convex boundary $\partial M$. If $u$ is a solution of the Dirichlet eigenvalue problem~\eqref{eq:1.0}, then
\begin{equation}\label{eq:2.11}
\int_{M}{\rm Ric}(\nabla u,\nabla u)|\nabla u|^{2p-4}dM<\lambda^{2}\widetilde{C}\int_{M}|\nabla u|^{2p-2}dM,
\end{equation}
where $\widetilde{C}:=\widetilde{C}(n,p,M)$ is a positive constant depending only $n$, $p$ and $M^{n}$.
\end{lemma}

\begin{proof}
From~\cite[Proposition 4.1]{Santos:23}, we have
\begin{equation}\label{eq:2.14}
\begin{split}
\int_{M}{\rm Ric}&(\nabla u,\nabla u)|\nabla u|^{2p-4}dM\\
&=\int_M\left((\Delta_p u)^2-|\nabla u|^{2p-4}|{\rm Hess}\,u|^2\right)dM+\int_{\partial M}|\nabla u|^{2p-4}\mathcal{Q}(u)d\sigma\\
&\quad-(p-2)\int_M|\nabla u|^{2p-6}\left((p-2)(\Delta_{\infty}u)^2|\nabla u|^{2}+2|{\rm Hess}\,u(\nabla u)|^2\right)dM,
\end{split}
\end{equation}
where $|\nabla u|^{2}\Delta_{\infty}u=\langle{\rm Hess}\,u(\nabla u),\nabla u\rangle$ and $\mathcal{Q}(u)$ is given by
\begin{equation}\label{eq:2.12}
\mathcal{Q}(u)=du(\eta)\left(\Delta^{\partial}u+(n-1)\mathcal{H}\,du(\eta)\right)+\langle\mathcal{A}_{\eta}(\nabla^{\partial}u),\nabla^{\partial}u\rangle+\langle\nabla^{\partial}u,\nabla^{\partial}du(\eta)\rangle.
\end{equation}
with $du(\eta)=\partial\eta/\partial u$.

On the one hand, since $p<2$, from Cauchy-Schwarz's inequality, \eqref{eq:2.14} reads
\begin{equation}\label{eq:3.32}
\begin{split}
\int_{M}|\nabla u|^{2p-4}{\rm Ric}(\nabla u,\nabla u)dM&\leq-(2p-3)\int_{M}|\nabla u|^{2p-4}|{\rm Hess}\,u|^{2}dM\\
&\quad+\int_{M}(\Delta_{p}u)^{2}dM-(p-2)^{2}\int_{M}|\nabla u|^{2p-4}(\Delta_{\infty}u)^{2}dM\\
&\qquad+\int_{\partial M}|\nabla u|^{2p-4}\mathcal{Q}(u)d\sigma.
\end{split}
\end{equation}
By using that $p>3/2$, we can use again the Cauchy-Schwarz's inequality in order to
\begin{equation}\label{eq:3.32.1}
-(2p-3)|\nabla u|^{2p-4}|{\rm Hess }u|^2\leq-\dfrac{(2p-3)}{n}(\Delta_p u)^2
\end{equation}
with equality holding if and only if
\begin{equation}\label{eq:2.16}
(p-2)|\nabla u|^{p-4}\langle{\rm Hess}\,u(\nabla u),X\rangle\nabla u+|\nabla u|^{p-2}{\rm Hess}\,u(X)=-\dfrac{1}{n}(\Delta_{p}u)I(X),
\end{equation}
for all $X\in\mathfrak{X}(M)$. Thus, by inserting~\eqref{eq:3.32.1} in~\eqref{eq:3.32}, we get
\begin{equation}
\begin{split}
\int_{M}|\nabla u|^{2p-4}{\rm Ric}(\nabla u,\nabla u)dM&\leq\dfrac{n-2p+3}{n}\int_{M}(\Delta_p u)^2dM-(p-2)^{2}\int_{M}|\nabla u|^{2p-4}(\Delta_{\infty}u)^{2}dM\\
&\quad+\int_{\partial M}|\nabla u|^{2p-4}\mathcal{Q}(u)d\sigma\\
&\leq\dfrac{n-2p+3}{n}\int_{M}(\Delta_p u)^2dM+\int_{\partial M}|\nabla u|^{2p-4}\mathcal{Q}(u)d\sigma
\end{split}
\end{equation}
Therefore, from~\eqref{eq:1.1}
\begin{equation}
\int_{M}|\nabla u|^{2p-4}{\rm Ric}(\nabla u,\nabla u)dM\leq\left(\dfrac{n-2p+3}{n}\right)\lambda^{2}\int_{M}|u|^{2p-2}dM+\int_{\partial M}|\nabla u|^{2p-4}\mathcal{Q}(u)d\sigma,
\end{equation}
with equality if and only if $\Delta_{\infty}u=0$. 

On the other hand, since $u=0$ on $\partial M$ and the boundary is mean convex, we can estimate the $\mathcal{Q}$ as follows
\begin{equation}
\int_{\partial M}|\nabla u|^{2p-4}\mathcal{Q}(u)d\sigma\leq0,
\end{equation}
and then
\begin{equation}\label{eq:3.35}
\int_{M}|\nabla u|^{2p-4}{\rm Ric}(\nabla u,\nabla u)dM\leq\lambda^{2}\left(\dfrac{n-2p+3}{n}\right)\int_{M}|u|^{2p-2}dM.
\end{equation}

Since $p\in(3/2,2)$, we have $1<2p-2<p$. Consequently, by Sobolev's embedding theory, $W^{1,p}_{0}(M)\subset W^{1,2p-2}_{0}(M)$. As $2p-2<(2p-2)^{\ast}$, from H\"older inequality,
\begin{equation}\label{eq:3.36.1}
\int_{M}|u|^{2p-2}dM\leq{\rm vol}(M)^{\frac{2p-2}{r}}\left(\int_{M}|u|^{(2p-2)^{\ast}}dM\right)^{\frac{2p-2}{(2p-2)^{\ast}}}
\end{equation}
for $r>1$. Hence, by using Gagliardo-Nirenberg-Sobolev's inequality for compact Riemannian manifolds with boundary~\cite{Aubin:12} (see also~\cite{Hebey:00}), we have
\begin{equation}\label{sobolev}
\int_{M}|u|^{2p-2}dM\leq{\rm vol}(M)^{\frac{2p-2}{r}}C_{1}\int_{M}|\nabla u|^{2p-2}dM,
\end{equation}
where $C_{1}$ is a constant positive depending on $M^{n}$. By inserting this in~\eqref{eq:3.35},
\begin{equation}\label{eq:3.35.1}
\int_{M}|\nabla u|^{2p-4}{\rm Ric}(\nabla u,\nabla u)dM\leq\lambda^{2}\widetilde{C}\int_{M}|\nabla u|^{2p-2}dM,
\end{equation}
where $\widetilde{C}$ is a positive constant depending only $n$, $p$ and $M^{n}$.

By assuming the equality case, then all inequalities become into equalities, in particular the equality~\eqref{eq:3.36.1} implies in the equality in H\"older inequality. Hence, $|u|=1$ almost everywhere and then the $p$-Laplacian of $u$ is identically zero, which cannot happens since the first eigenvalue if nonzero. Therefore, the inequality~\eqref{eq:3.35.1} is not sharp.
\end{proof}


As a first application, we get the following lower estimate:
\begin{theorem}\label{principal}
Let $M^{n}$ be a compact Riemannian manifold with mean convex boundary and consider $M^{n}$ as a minimal submanifold of $\mathbb{S}^{n+q}$ with $S=const$. Then
\begin{equation}
\lambda_{1,p}(M)>\sqrt{(n-S)C},\qquad p\in(3/2,2),
\end{equation}
where $C$ is a positive constant depending only $n$, $p$ and $M^{n}$.
\end{theorem}

\begin{proof}
From Lemma~\ref{lem:2.2},
\begin{equation}\label{eq:2.11.1}
\int_{M}{\rm Ric}(\nabla u,\nabla u)|\nabla u|^{2p-4}dM<\widetilde{C}\int_{M}|\nabla u|^{2p-4}dM.
\end{equation}

On the other hand, since $M^{n}$ is a minimal submanifold of $\mathbb{S}^{n+q}$, by taking $H=0$ in Lemma~\ref{lem:2.1},
\begin{equation}
{\rm Ric}(\nabla u,\nabla u)\geq-\dfrac{n-1}{n}\left(S-n\right)|\nabla u|^{2}.
\end{equation}
By replacing this in~\eqref{eq:2.11},
\begin{equation}\label{eq:2.11.2}
-\dfrac{n-1}{n}\left(S-n\right)\int_{M}|\nabla u|^{2p-2}dM<\lambda^{2}\widetilde{C}\int_{M}|\nabla u|^{2p-2}dM
\end{equation}
where was used that $S=const.$. Hence
\begin{equation}\label{eq:2.11.3}
0<\left(\lambda^{2}\widetilde{C}+\dfrac{n-1}{n}\left(S-n\right)\right)\int_{M}|\nabla u|^{2p-2}dM
\end{equation}
and consequently since $\int_{M}|\nabla u|^{2p-2}dM>0$, we have
\begin{equation}\label{eq:2.11.4}
\lambda^{2}>\dfrac{n}{(n-1)\widetilde{C}}\left(n-S\right).
\end{equation}
In particular,
\begin{equation}
\lambda_{1,p}(M)^{2}>(n-S)C,
\end{equation}
and we have the desired result.
\end{proof}

By assuming that submanifold has prescribed constant scalar curvature, we get
\begin{theorem}\label{thm:2.5}
Let $M^{n}$ be a compact Riemannian manifold with constant scalar curvature $R=1$ and mean convex boundary. Let us consider $M^{n}$ as a submanifold of $\mathbb{S}^{n+q}$ with $S=const.$. Then
\begin{equation}
\lambda_{1,p}(M)>\sqrt{(n^{2}-2(n-2)S)C},\qquad p\in(3/2,2),
\end{equation}
where $C$ is a positive constant depending only $n,p$ and $M^{n}$.
\end{theorem}

\begin{proof}
Since $R=1$, from~\eqref{eq:2.7} we have $S=n^{2}H^{2}$. By Lemma~\ref{lem:2.1},
\begin{equation}
\begin{split}
{\rm Ric}&\geq-\dfrac{n-1}{n}\left(S+\frac{n(n-2)}{\sqrt{n(n-1)}}H\sqrt{S-nH^2}-n-2nH^2\right)\\
&\geq-\dfrac{n-1}{n}\left(S+n^{2}H^{2}-n-4nH^2\right)\\
&\geq-\dfrac{n-1}{n^{2}}\left(2(n-2)S-n^{2}\right).
\end{split}
\end{equation}
Hence, from Lemma~\ref{lem:2.2},
\begin{equation}\label{eq:2.11.5}
0<\left(\lambda^{2}\widetilde{C}+\dfrac{n-1}{n^{2}}\left(2(n-2)S-n^{2}\right)\right)\int_{M}|\nabla u|^{2p-2}dM.
\end{equation}
So
\begin{equation}\label{eq:2.11.6}
\lambda^{2}>\dfrac{n-1}{n^{2}\widetilde{C}}(n^{2}-2(n-2)S),
\end{equation}
and therefore,
\begin{equation}\label{eq:2.11.7}
\lambda_{1,p}(M)>\sqrt{(n^{2}-2(n-2)S)C}.
\end{equation}
\end{proof}

We end this paper with the following low estimate to $\lambda_{1,p}(M)$ involving the squared norm of the total umbilicity tensor and without the necessary of prescribed the scalar curvature.
\begin{theorem}\label{thm:2.6}
Let $M^{n}$ be a compact Riemannian manifold having mean convex boundary and consider $M^{n}$ as a submanifold of $\mathbb{S}^{n+q}$ with $S-nH^{2}=const.$. Then
\begin{equation}\label{eq:4.31}
\lambda_{1,p}(M)>\sqrt{\left(4(n-1)-n(S-nH^{2})\right)C},\qquad p\in(3/2,2),
\end{equation}
where $C$ is a positive constant depending only $n,p$ and $M^{n}$.
\end{theorem}

\begin{proof}
Following~\cite{Santos:24}, by using $\varepsilon$-Young's inequality $2xy\leq\varepsilon x^{2}+\varepsilon^{-1}y^{2}$ for
\begin{equation*}\label{eq:4.23}
x=H\quad\mbox{and}\quad y=\sqrt{S-nH^{2}}
\end{equation*}
we have
\begin{equation}\label{eq:4.24}
2H\sqrt{S-nH^{2}}\leq\varepsilon H^{2}+\varepsilon^{-1}(S-nH^{2}).
\end{equation}
Consequently, by taking $\varepsilon=\frac{2\sqrt{n(n-1)}}{n-2}$, \eqref{eq:4.24} reads
\begin{equation*}\label{eq:4.25}
\dfrac{n(n-2)}{\sqrt{n(n-1)}}H\sqrt{S-nH^{2}}\leq nH^{2}+\dfrac{(n-2)^{2}}{4(n-1)}(S-nH^{2}).
\end{equation*}
Hence, inserting this in Lemma~\ref{eq:2.11} we reach at the following inequality
\begin{equation}\label{eq:4.26}
\begin{split}
\mathrm{Ric}(\nabla u,\nabla u)&\geq-\dfrac{n-1}{n}\left(S+nH^{2}+\dfrac{(n-2)^{2}}{4(n-1)}(S-nH^{2})-n-2nH^{2}\right)|\nabla u|^2\\
&=\dfrac{n-1}{n}\left(n-\dfrac{n^{2}}{4(n-1)}(S-nH^{2})\right)|\nabla u|^2.
\end{split}
\end{equation}
By replacing in Lemma~\ref{lem:2.2},
\begin{equation}\label{eq:2.11.8}
(n-1)\left(1-\dfrac{n}{4(n-1)}(S-nH^{2})\right)\int_{M}|\nabla u|^{2p-2}dM<\lambda^{2}\widetilde{C}\int_{M}|\nabla u|^{2p-2}dM
\end{equation}
and so,
\begin{equation}\label{eq:2.11.9}
0<\lambda^{2}\widetilde{C}-\dfrac{1}{4}\left(4(n-1)-n(S-nH^{2})\right).
\end{equation}
Hence
\begin{equation}
\lambda_{1,p}(M)^{2}>\left(4(n-1)-n(S-nH^{2})\right)C.
\end{equation}
\end{proof}

\begin{remark}
It should notice that, by applying a similar argument, the Theorems~\ref{thm:2.5} and~\ref{thm:2.6} also holds for closed manifolds. In fact, the H\"older inequality used in~\eqref{eq:3.36.1} does not depends of the boundary. Moreover, the Sobolev inequality~\eqref{sobolev} is still true for closed manifolds (see~\cite[Theorem 4.5]{Hebey:00}).
\end{remark}

\subsection*{Acknowledgment}
The first author is partially supported by CNPq, Brazil, grant 311124/2021-6 and Propesqi (UFPE). The second author is partially supported by CNPq, Brazil.


\begin{thebibliography}{n}
		
\bibitem{Aubin:12} T. Aubin, 
{\em Nonlinear analysis on manifolds. Monge-Ampère equations}, 
Springer Science \& Business Media, 2012.		

\bibitem{Benedikt:18} J. Benedikt, P. Girg, L. Kotrla and P. Takac,
\textit{Origin of the p-Laplacian and A. Missbach}, 
Electron. J. Differential Equations. \textbf{16} (2018), pp. 1--17.
				
\bibitem{Dama:98} L. Damascelli and F. Pacella, 
{\em Monotonicity and symmetry of solutions of  $p$-Laplace equations, $1< p< 2$, via the moving plane method}, 
Annali Della Scuola Normale Superiore Di Pisa-Classe Di Scienze. \textbf{26} (1998), pp. 689--707.
	
\bibitem{Diaz:94} Jl. Diaz and F. de Thelin, 
{\em On a nonlinear parabolic problem arising in some models related to turbulent flows}, 
SIAM J. on Math. Anal. \textbf{25} (1994), 1085--1111.
					
\bibitem{Santos:23} F.R. dos Santos \& M.N. Soares,
\textit{Lower bounds for the length of the second fundamental form via the first eigenvalue of the p-Laplacian}, 
Nonlinear Analysis. \textbf{232} (2023), pp. 113251.
		
\bibitem{Santos:24} F.R. dos Santos \& M.N. Soares,
\textit{A Reilly type integral inequality for the $p$-Laplacian and applications to the unit sphere},
to appear in Rev. R. Acad. Cien. Exactas Fís. Nat. Ser. A Mat. RACSAM (2024).

\bibitem{Hebey:00} E. Hebey,
{\em Nonlinear Analysis on Manifolds: Sobolev Spaces and Inequalities},
Courant lecture notes in mathematics. American Mathematical Soc., 2000.

\bibitem{le:06} A. Lê, 
{\em Eigenvalue problems for the p-Laplacian},
Nonlinear Analysis: Theory, Methods and Applications \textbf{64} (2006), 1057--1099.

\bibitem{Ledoux:99} M. Ledoux, 
{\em On manifolds with non-negative Ricci curvature and Sobolev inequalities}, 
Comm. In Anal. and Geo. \textbf{7} (1999), pp. 347--353.

\bibitem{Leibenson:45} L.S. Leibenson, 
{\em General problem of the movement of a compressible fluid in a porous medium}, 
Bull. Acad. Sci. URSS. Sér. Géograph. Géophys. [Izvestia Akad. Nauk SSSR]. \textbf{9} (1945), pp. 7--10.

\bibitem{Leung:92} P.F. Leung,  
{\em An estimate on the Ricci curvature of a submanifold and some applications}, 
Proc. Amer. Math. Soc. {\bf114} (1992), 1051--1061.
		
\bibitem{Leung:83} P.F. Leung,
{\em Minimal submanifolds in a sphere},
Math. Z. {\bf183} (1983), 75--86	.
		
\bibitem{Lindqvist:15} P. Lindqvist,
{\em Notes on the Infinity-Laplace equation},
Norwegian University of Science and Technology, 2015.
		
\bibitem{Lindqvist:12} P. Lindqvist, 
{\em On the equation $\mathrm{div}(|\nabla u|^{p-2} \nabla u) + \lambda|u|^{p-2} u = 0$}, 
Proc. Am. Math. Soc. \textbf{109} (1990), pp. 157--164.
		
\bibitem{Liu:07} J. Liu and Q. Zhang, 
{\em Simons-type inequalities for the compact submanifolds in the space of constant curvature}, 
Kodai Mat. J. \textbf{30} (2007), pp. 344--351.
				
\bibitem{Matei:00} A. Matei,
{\em First eigenvalue for the $p$-Laplace operator},
Nonlinear. Anal. Theory, Methods and Appl. {\bf39} (2000), pp. 1051--1068.
		
\bibitem{Wei:11} G. Wei, 
{\em Some applications of Bochner formula to submanifolds of a unit sphere}, 
Pub. Math. Deb. \textbf{78}  (2011), pp. 625--631.
		
\end{thebibliography}
\end{document}